\documentclass[a4paper,12pt]{article}
\usepackage[margin=20truemm]{geometry}

\usepackage{amsthm}
\usepackage{amsmath}
\usepackage{amssymb}
\usepackage{amsfonts}
\usepackage{latexsym}
\usepackage{mathrsfs}
\usepackage{mathtools} 

\usepackage[T1]{fontenc}
\usepackage{lmodern} 
\usepackage{bm} 

\numberwithin{equation}{section}

\allowdisplaybreaks[1]

\usepackage{enumitem}

\newcommand{\down}{\ensuremath{\downarrow}}

\newcommand{\ci}{\mathrm{i}}
\newcommand{\ce}{\mathrm{e}}
\newcommand{\cd}{\mathrm{d}}

\DeclarePairedDelimiter{\abs}{\lvert}{\rvert} 
\DeclarePairedDelimiter{\rbra}{(}{)} 
\DeclarePairedDelimiter{\cbra}{\{}{\}} 
\DeclarePairedDelimiter{\sbra}{[}{]} 

\newcommand{\bP}{\ensuremath{\mathbb{P}}}

\newcommand{\bR}{\ensuremath{\mathbb{R}}}

\newcommand{\bW}{\ensuremath{\mathbb{W}}}

\newcommand{\cD}{\ensuremath{\mathcal{D}}}

\newcommand{\cF}{\ensuremath{\mathcal{F}}}

\newcommand{\cR}{\ensuremath{\mathcal{R}}}

\theoremstyle{plain}
\newtheorem{Thm}{Theorem}[section]
\newtheorem{Lem}[Thm]{Lemma}

\theoremstyle{definition}

\newtheorem{Rem}[Thm]{Remark}
\theoremstyle{remark}

\theoremstyle{plain}
\newtheorem*{Thm*}{Theorem}
\newtheorem*{Lem*}{Lemma}
\newtheorem*{Prop*}{Proposition}
\newtheorem*{Cor*}{Corollary}
\newtheorem*{Conj*}{Conjecture}
\theoremstyle{definition}
\newtheorem*{Ass*}{Assumption}
\newtheorem*{Def*}{Definition}
\newtheorem*{Rem*}{Remark}
\theoremstyle{remark}
\newtheorem*{Eg*}{Example}

\usepackage[%
bookmarks=true,%
bookmarksdepth=3,%
bookmarksnumbered=true,%
colorlinks=false,%
setpagesize=false,%
pdftitle={},%
pdfsubject={},%
pdfauthor={},%
pdfkeywords={}%
]{hyperref}

\title{\textbf{Supremum penalizations for L\'{e}vy processes}}
\author{Shosei Takeda\footnote{Rakunan High School, Kyoto, Japan. (email:
\texttt{takeda.shosei@gmail.com})}}
\date{}

\begin{document}
\maketitle
\begin{abstract}
  Several long-time limit theorems of one-dimensional Lévy processes weighted
  and normalized by functions of its supremum are studied.
  The long-time limits are taken via the families of exponential times and
  that of constant times, called exponential clock and constant clock,
  respectively.
\end{abstract}
{\small Keywords and phrases: one-dimensional L\'{e}vy process;
penalization; limit theorem; conditioning }

\section{Introduction}
Roynette--Vallois--Yor~\cite{MR2253307, MR2229621} (see also~\cite{MR2261065,MR2504013})
have considered the exsitence of the following limit distribution weighted by
a certain adapted process \(\Gamma\)
for a Brownian motion,
which they called the \textit{penalization problem}:
\begin{align}
  \bW^\Gamma \coloneqq \lim_{s\to\infty} \frac{\Gamma_s}{\bW\sbra{\Gamma_s}} \cdot \bW,
  \label{eq:BM-penal}
\end{align}
where \((X, \bW)\) denotes the canonical representation of a one-dimensional
standard Brownian motion starting from \(0\).
Roynette--Vallois--Yor have considered many kinds of Brownian penalizations as follows:
\begin{enumerate}
  \item \textit{supremum penalization:} \( \Gamma_s = f(S_s)\),
  where \(S_s =\sup_{u\le s} X_u\);
  \item \textit{local time penalization:} \(\Gamma_s = f(L_s)\), where \(L_s\)
  denotes the local time at \(0\) of \(X\);
  \item \textit{Kac killing penalization:} \(\Gamma_s=\exp(-\int_0^s q(X_u) \, \cd u)\),
\end{enumerate}
for non-negative integrable functions \(f\) and \(q\).
In the present paper, we only discuss the supremum penalization.

Roynette--Vallois--Yor~\cite[Theorem 3.6]{MR2253307} have obtained the
following limit property called supremum penalization: for
any bounded adapted functional \(F_t\), and for any
non-negative integrable function \(f\) which satisfies
\(\int_0^\infty f(x)\, \cd x>0\), it holds that
\begin{align}\label{eq:def-Wf}
  \lim_{s\to\infty} \frac{\bW\sbra{F_t f(S_s)}}{\bW\sbra{f(S_s)}}
  = \bW\sbra*{F_t \frac{M^{(f)}_t}{M^{(f)}_0}} \eqqcolon \bW^{(f)}\sbra{F_t},
  \quad t\ge 0,
\end{align}
where \(M^{(f)}\) is the non-negative martingale given by
\begin{align}
  M^{(f)}_t = f(S_t)(S_t-X_t) + \int_{S_t}^\infty f(u)\, \cd u,
  \quad t\ge 0.
\end{align}
The class of these martingales \((M_t^{(f)})\) is known as (Brownian)
A\'{z}ema--Yor martingales.
Under the characteristic probability measure \(\bW^{(f)}\), called
\textit{penalized measure}, the supremum of the process \(X\)
is finite almost surely, i.e., \(\bW^{(f)}(S_\infty<\infty) = 1\).
In particular, the penalized measure \(\bW^{(f)}\) is singular to \(\bW\).
Let \(g=\sup\cbra{t\ge 0\colon X_t=S_\infty}\) denote the
last hitting time of the maximum.
In~\cite[Theorem 4.6]{MR2253307}, the authors have also obtained the
following sample path behaviors under the measure \(\bW^{(f)}\):

\begin{Thm}[{\cite[Theorem 4.6]{MR2253307}}]\label{Thm:Wf}
  Under \(\bW^{(f)}\), the following assertions hold:
  \begin{enumerate}
    \item \(\bW^{(f)}(S_\infty \in \cd x)=f(x)\, \cd x/\int_0^\infty f(x)\, \cd x\)
          and \(\bW^{(f)}(g<\infty)=1\);
    \item the processes \((X_s, s\le g)\) and \((X_g - X_{g+s}, s\ge 0)\) are
          independent;
    \item given \(S_\infty = x>0\), the process \((X_s, s\le u)\)
          is distributed as a Brownian motion starting from \(0\) and
          stopped at its first hitting time of \(x\);
    \item the process
           \((X_g - X_{g+u}, u\ge 0)\) is distributed as a three-dimensional Bessel
    process starting from \(0\).
  \end{enumerate}
\end{Thm}
The result of the limit form~\eqref{eq:BM-penal} was generalized to
the strictly stable processes by Yano--Yano--Yor~\cite{MR2744885}.
Theorem~\ref{Thm:Wf} was generalized to the general L\'{e}vy processes
by Yano~\cite{MR3127911}; see Theorem~\ref{Thm:Pf}.
We have obtained the penalization result of the form~\eqref{eq:BM-penal}
for L\'{e}vy processes.

Let \(\cD[0,\infty)\) denote the space of c\`{a}dl\`{a}g paths:
\(\omega \colon [0, \infty) \to \bR \cup \cbra{\partial}\) with
lifetime \(\zeta(\omega) = \inf\cbra{s \colon \omega(s) = \partial}\)
where \(\partial\) is the cemetery point.
Let \((X, \bP)\) denote the canonical representation
of a one-dimensional L\'{e}vy process starting from \(0\)
on \(\cD[0,\infty)\).
For \(t\ge 0\), we denote by \(\cF_t^X = \sigma(X_s, s\le t)\) the natural
filtration of \(X\) and write \(\cF_t = \bigcap_{s>t} \cF_s^X\) and
\(\cF_\infty = \sigma\rbra*{\bigcup_{t\ge 0} \cF_t}\).
The L\'{e}vy process \((X, \bP)\) can be identified
using its characteristic exponent \(\varPsi\), defined by the equation
\(
  \bP\sbra{\ce^{\ci\lambda X_t}} = \ce^{-t \varPsi(\lambda)},
  \, t \ge 0, \lambda \in \bR,
\)
which is given by the L\'{e}vy--Khinchin formula
\begin{align}
  \varPsi(\lambda)
  = \ci \gamma \lambda
  + \frac{1}{2} \sigma^2 \lambda^2
  + \int_\bR \rbra*{1 - \ce^{\ci \lambda x} + \ci \lambda x 1_{\cbra{\abs{x}< 1}}}
  \nu(\cd x),\quad \lambda\in \bR
\end{align}
for some constants \(\gamma \in \bR,\, \sigma \ge 0\)
and some measure \(\nu\), called L\'{e}vy measure, on \(\bR\)
which satisfies \(\nu(\cbra{0})=0\) and
\(\int_\bR \rbra*{x^2 \wedge 1} \nu(\cd x) < \infty\).

We denote by \(S\) and \(I\) the past supremum and the past infimum process up to time
\(t < \zeta \), i.e.,
\begin{align}
  S_t = \sup_{s\le t} X_s
  \quad\text{and}\quad
  I_t =  \inf_{s\le t} X_s.
\end{align}
We denote the hitting time of a Borel set \(A \subset \bR\) of \(X\) by
\begin{align}
  T_A = \inf \cbra{t>0 \colon X_t \in A}.
\end{align}
We assume the following three conditions:
\begin{enumerate}[label=\textbf{(\Alph*)}]
  \item \(0\) is regular for both \((-\infty, 0)\) and \((0, \infty)\)
        with respect to \(X\), that is,
        \begin{align}
          \bP(T_{(-\infty, 0)} = 0) = \bP(T_{(0, \infty)}= 0 ) = 1.
        \end{align}\label{cond:A}
  \item The resolvent of \(X\) have the absolutely continuous density, that is,
        for every \(q > 0\), there exists an integrable function \(r_q\) which
        satisfies
        \begin{align}
          \bP\sbra*{\int_0^\infty \ce^{-qt} f(X_t) \,\cd t} = \int_{-\infty}^\infty
          f(x)
          r_q(x) \, \cd x,
        \end{align}
        for all non-negative Borel functions \(f\).\label{cond:B}
  \item \(X\) is an oscillating process, i.e.,
        \begin{align}
          \bP(S_\infty = \infty) = \bP(I_\infty = - \infty) = 1.
        \end{align}\label{cond:C}
\end{enumerate}
Let \(R = S - X\) denote the reflected process of \(X\) at the supremum.
Under the above conditions~\ref{cond:A}--\ref{cond:C}, \(0\) is regular for itself
with respect to \(R\),
and hence we can define a continuous local time \(L = \rbra{L_t, t \ge 0}\)
at \(0\) of \(R\), which is normalized by
\(\bP\sbra*{\int_0^\infty \ce^{-t}\, \cd L_t} =1\).
We denote by
\(\eta_t = \inf\{s \ge 0\colon L_s > t\}\)
the right-continuous inverse of \(L\), called \textit{inverse local time of \(L\)}.
We define \(H_t = S_{\eta_t}=X_{\eta_t}\). Then
\((\eta, H)\) is a bivariate subordinator and \(\eta\) and \(H\) are called
the ascending ladder time process and the ascending ladder height process of
\(X\), respectively;
see, e.g., Kyprianou~\cite[Section 6]{MR3155252}.
Let \(\kappa\) denote the Laplace exponent of \((\eta,H)\)
given by \(\bP\sbra{\ce^{-q\eta_t-\lambda H_t}}=\ce^{-t\kappa(q, \lambda)},\,
 q\ge 0,\,\lambda\ge 0, \,t\ge 0\).

We define, for \(q \ge 0\),
\begin{align}
  h_q(x) = \bP\sbra*{\int_0^\infty \ce^{-qt} 1_{\cbra{S_t \le x}} \, \cd L_t}
  =\bP\sbra*{\int_0^\infty \ce^{-q\eta_t} 1_{\cbra{H_t \le x}}\, \cd t},
  \quad x\ge 0,
\end{align}
and \(h \coloneqq h_0\).
Note that \(h\) is the potential function of \(H\).
Clearly, we have \(h_q(x)\ge 0\), \(h_q(0)=0\) and \(h_q\) is
non-decreasing.
By Silverstein~\cite{MR573292}, the function
\(h\) is continuous, subadditive, and differentiable;
see also Chaumont~\cite[Theorem 2]{MR3098676}.
If \((X,\bP) \) is a standard Brownian motion, we have \(h(x)=x\).
If \((X, \bP)\) is a strictly \((\alpha, \rho)\)-stable process,
we have \(h(x)=x^{\alpha\rho}\), where \(\alpha \in (0,2)\) is the index
and \(\rho = \bP(X_t \ge 0)\in [1-1/\alpha, 1/\alpha] ,\, t\ge 0\)
is called the positivity parameter.
We can calculate the Laplace transform of \(h_q\): for \(q\ge 0\),
\begin{align}
  \lambda \int_0^\infty \ce^{-\lambda x}h_q(x)\, \cd x
  = \int_0^\infty \bP\sbra{\ce^{-q\eta_t-\lambda H_t}}\, \cd t
  = \frac{1}{\kappa(q, \lambda)},\quad \lambda>0.\label{eq:laplace-hq}
\end{align}
Let \(f\) be a non-negative Borel function on \([0, \infty)\) satisfying
\begin{align}\label{cond:f}
  0 < \int_0^\infty f(x)h^\prime(x) \,\cd x < \infty.
\end{align}
Then, we introduce the process given by
\begin{align}
  M_t^{(f)} =f(S_t)h(S_t - X_t) + \int_{S_t}^\infty f(x)h^{\prime}(x - X_t)
  \,\cd x, \quad t\ge 0.
\end{align}

\begin{Thm}[{\cite[Theorems 6.1 and 7.6]{MR3127911}}]\label{Thm:AY-mart}
  The process \((M_t^{(f)}, t\ge 0)\) is a \(\rbra{(\cF_t), \bP}\)-martingale
  and satisfies \(M_t^{(f)} \to 0\), \bP-a.s.\ as \(t \to \infty\).
\end{Thm}
The martingale \(M^{(f)}\) is called a \textit{generalized A\'{z}ema--Yor martingale}.
Let us define new probability measure in the same way as~\eqref{eq:def-Wf} as follows:
\begin{align}
  \bP^{(f)} |_{\cF_t} = \frac{M_t^{(f)}}{M_0^{(f)}} \cdot \bP|_{\cF_t},\quad t\ge 0.
\end{align}
Since \(M^{(f)}\) is a martingale, the probability measure
\(\bP^{(f)}\) is well-defined on \(\cF_\infty\).
Yano~\cite[Theorem 7.5]{MR3127911} have obtained the following sample path
behaviors under \(\bP^{(f)}\):
\begin{Thm}[{\cite[Theorem 7.5]{MR3127911}}]\label{Thm:Pf}
  Under \(\bP^{(f)}\), the following assertions hold:
  \begin{enumerate}
    \item \(\bP^{(f)}(S_\infty \in \cd x)=f(x)h'(x)\, \cd x/M_0^{(f)}\)
    and \(\bP^{(f)}(g<\infty)=1\);
    \item the processes \((X_s, s\le g)\) and \((X_g - X_{g+s}, s\ge 0)\) are
          independent;
    \item given \(S_\infty = x>0\), the process \((X_s, s\le u)\)
          is distributed as a process starting from \(0\) and
          stopped at its first hitting time of \(x\);
    \item the process
           \((X_g - X_{g+u}, u\ge 0)\) is distributed as the process conditioned to
          stay negative; see~\cite{MR3127911} (see also~\cite{MR2164035} ).
  \end{enumerate}
\end{Thm}
Let \(\bm{e}\) be an exponential distributed random variable
with mean \(1\) which is independent of \(\cF_\infty\),
and we write \(\bm{e}_q=\bm{e}/q\) for \(q> 0\).
We have obtained the following penalization result for L\'{e}vy processes:
\begin{Thm}[Penalization with exponential clock]\label{Thm:main-exp}
  Suppose the conditions~\ref{cond:A}--\ref{cond:C} hold.
  Let \(f\) be a non-negative Borel function satisfying~\eqref{cond:f}. Define,
  for \(q>0\),
  \begin{align}
    N_t^{(q, f)} = \frac{\bP\sbra{f(S_{\bm{e}_q}); t\le \bm{e}_q| \cF_t}}{\kappa(q,0)}
    \quad \text{and} \quad
    M_t^{(q, f)} = \frac{\bP\sbra{f(S_{\bm{e}_q})| \cF_t}}{\kappa(q,0)},
    \quad t\ge 0.
  \end{align}
  Then it holds that
  \begin{align}
    \lim_{q\down 0} N_t^{(q, f)} = \lim_{q\down 0} M_t^{(q, f)} = M_t^{(f)}, \quad
    \text{\(\bP\)-a.s.\ and in \(L^1(\bP)\).}
  \end{align}
  Consequently, it holds that
  \begin{align}\label{eq:exp-penal}
    \lim_{q\down 0}\frac{\bP\sbra{F_t f(S_{\bm{e}_q})}}{\bP\sbra{f(S_{\bm{e}_q})}}
    = \bP\sbra*{F_t\frac{M_t^{(f)}}{M_0^{(f)}}} = \bP^{(f)}\sbra{F_t}.
  \end{align}
  for all bounded \(\cF_t\)-measurable functionals \(F_t\).
\end{Thm}
The proof of Theorem~\ref{Thm:main-exp} will be given in Section~\ref{Sec:exp-clock}.
Theorem~\ref{Thm:main-exp} says that
the measure \(\bP^{(f)}\) can be obtained as the limit
via the family of exponential random variables \((\bm{e}_q)\) as \(q\down 0\),
instead of the family of constant times \((t)\) as \(t\to\infty\).
We call \((\bm{e}_q)\) the \textit{exponential clock}.
This approach already appeared on other penalizations. For example,
we refer to Takeda--Yano~\cite{10.1214/23-EJP903}
for local time penalization for L\'{e}vy processes,
Profeta--Yano--Yano~\cite{MR3909919} for local time penalization for diffusions.

We have also obtained a penalization result via constant clock under the following
additional assumptions:
\begin{enumerate}[label=\textbf{(D\arabic*)}]
  \item The transition semigroup of \((X, \bP)\) is absolutely continuous
        and there is a version of its densities,
        denoted by \(x \mapsto p_t(x), \, x\in \bR\),
        which are bounded for all \(t>0\);\label{cond:D1}
  \item For any \(c \in \bR\), the process \((\abs{X_t - ct}, t\ge 0)\)
        is not a subordinator;\label{cond:D2}
  \item \(\lim_{t\to\infty}\bP(X_t \ge 0) = \rho \in (0, 1)\).
        \label{cond:regularity}\label{cond:D3}
\end{enumerate}

The condition~\ref{cond:regularity} is regular variation condition.
This condition~\ref{cond:regularity} is equivalent to each of the following conditions:
\begin{enumerate}
  \item[\textbf{(D3)\('\)}] Spitzer's condition:
   \(\lim_{t\to\infty}\frac{1}{t} \int_0^t \bP(X_s \ge 0) \,
        \cd s = \rho \in (0, 1)\);
  \item[\textbf{(D3)\(''\)}] \(t \mapsto n(t< \zeta) \in \cR_\infty(-\rho), \, \rho \in (0,1)\),
        where \(\cR_\infty(-\rho)\) denotes the functions
        which are regularity varying at infinity with index \(-\rho\),
        i.e.,
        \begin{align}
          \lim_{t\to\infty}\frac{n(\lambda t<\zeta)}{n(t<\zeta)}=\lambda^{-\rho},
          \quad \lambda>0.
        \end{align}
\end{enumerate}
For more details, see Theorem VI.14 in~\cite{MR1406564} and Theorem 2
in~\cite{MR1443955}.

\begin{Thm}[Penalization with constant clock]\label{Thm:const-result}
  Suppose the conditions~\ref{cond:A}--\ref{cond:C} and~\ref{cond:D1}--\ref{cond:D3}
  hold.
  Let \(f\) be a non-negative Borel function satisfying
  \(\int_0^\infty f(x) (h(x) \vee h^\prime(x)) \,\cd x< \infty\).
  Define, for \(s\ge 0\),
  \begin{align}
    M_t^{(s, f)} =\frac{\bP\sbra{f(S_s)|\cF_t}}{n(s<\zeta)},\quad t\ge 0.
  \end{align}
   Then, it holds that
   \begin{align}
    \lim_{s\to\infty}  M_t^{(s, f)} = M_t^{(f)}, \quad
    \text{\(\bP\)-a.s.\ and in \(L^1(\bP)\).}
   \end{align}
  Consequently,
  it holds that
  \begin{align}
    \lim_{s\to\infty} \frac{\bP\sbra{F_t f(S_s)}}{\bP\sbra{f(S_s)}}
    = \bP\sbra*{F_t\frac{M_t^{(f)}}{M_0^{(f)}}} = \bP^{(f)}\sbra{F_t},
  \end{align}
  for all bounded \(\cF_t\)-measurable functionals \(F_t\).
\end{Thm}
The proof of Theorem~\ref{Thm:const-result} will be given
in Section~\ref{Sec:const-clock}.
\begin{Rem}
  Whether Theorem~\ref{Thm:const-result} holds also for
  \(f\) satisfying \(\int_0^\infty f(x)h^\prime (x) \, \cd x < \infty \)
  is an open problem.
\end{Rem}

\subsection*{Organization}
The remainder of this paper is organized as follows.
In Section~\ref{Sec:preliminaries}, we prepare
certain properties and
preliminary facts of L\'{e}vy processes.
In Section~\ref{Sec:exp-clock}, we prove Theorem~\ref{Thm:main-exp}.
In Section~\ref{Sec:const-clock}, we prove Theorem~\ref{Thm:const-result}.
\section{Preliminaries}\label{Sec:preliminaries}
It is known that
\begin{align}
  \lambda \int_0^\infty \ce^{-\lambda x} \bP(S_{\bm{e}_q}\le x)\, \cd t
  = \bP\sbra{\ce^{-\lambda S_{\bm{e}_q}}}
  = \frac{\kappa(q, 0)}{\kappa(q, \lambda)},
  \quad q>0,\, \lambda>0;\label{eq:laplace-Psq}
\end{align}
see, e.g.,~\cite[Section VI.2]{MR1406564} and~\cite[Theorem 6.15]{MR3155252}.
Comparing~\eqref{eq:laplace-hq} and~\eqref{eq:laplace-Psq}, we obtain
\begin{align}
  \bP(S_{\bm{e}_q}\le x) = \kappa(q, 0)h_q(x),\quad q>0, \, x\ge 0.\label{eq:PSeq}
\end{align}
We denote by
\begin{align}
  V(\cd s, \cd x) = \int_0^\infty \bP(\eta_t\in \cd s, H_t\in \cd x)\, \cd t,
  \quad s\ge 0,\, x\ge 0
\end{align}
the potential measure of the subordinator \((\eta, H)\), and we write
\(V(\cd s, x)= V(\cd s, [0,x])\).
Then, we have \(V([0,\infty), \cd x) = h'(x)\, \cd x\) and \(V([0,\infty), x)=h(x)\).

Let \(n\) denote the characteristic measure of excursions away from \(0\)
of the process \(R\).
Then, the L\'{e}vy measure of \((\eta , H)\) is
\(n(\zeta \in \cd s, S_\zeta \in \cd x)\),
and its Laplace exponent is given by
\begin{align}
  \kappa(q, \lambda) = \gamma_H \lambda
  +\int_{[0,\infty)^2} (1-\ce^{-qs-\lambda x})
  n(\zeta \in \cd s, S_\zeta \in \cd x),
\end{align}
where \(\gamma_H\ge 0\) represents the drift of \(H\);
see, e.g.,~\cite[Section 6]{MR3155252}.
By~\cite[Theorem 2]{MR3098676},
the conditions~\ref{cond:A} and~\ref{cond:B} implies that \(\bP(S_t\in \cd x)\)
is absolutely continuous on \(x\ge 0\).
Let us denote its density by \(\varphi_t(x)\).
Then, by~\cite[(5,18)]{MR3098676}, it holds that
\begin{align}
  \varphi_t(x)\, \cd x=\bP(S_t\in \cd x) = \int_0^t
   n(t-s<\zeta) V(\cd s, \cd x), \quad
  x\ge 0.\label{eq:PSt}
\end{align}
The equation~\eqref{eq:PSeq} also follows from~\eqref{eq:PSt};
in fact, it holds that, for \(q>0\)
\begin{align}
  \bP(S_{\bm{e}_q}\in \cd x)
  &=q \int_0^\infty \ce^{-qs}n(s<\zeta)\, \cd s
  \int_0^\infty \ce^{-qs} V(\cd s, \cd x) \\
  &= \int_0^\infty (1-\ce^{-qu}) n(\zeta\in \cd u)
  \int_0^\infty \ce^{-qs} V(\cd s, \cd x) \\
  &= \kappa(q, 0) \int_0^\infty \ce^{-qs} V(\cd s, \cd x),\label{eq:PSeq-dx}
\end{align}
and
\begin{align}
  \bP(S_{\bm{e}_q}\le x) = \kappa(q,0)\int_0^\infty \ce^{-qs} V(\cd s, x)
  = \kappa(q,0)h_q(x).
\end{align}

\section{The result with exponential clock}\label{Sec:exp-clock}

\begin{proof}[Proof of Theorem~\ref{Thm:main-exp}]
  By the Markov property, we have
  \begin{align}
     \bP\sbra{f(S_{\bm{e}_q});t \le \bm{e}_q|\cF_t}
     = \ce^{-qt} \widetilde{\bP}\sbra{f((\widetilde{S}_{\widetilde{\bm{e}}_q}+X_t)
     \vee S_t)},
  \end{align}
  where the symbol \(\widetilde{\phantom{X}}\) means independence.
  Hence we have
  \begin{align}
    \bP\sbra{f(S_{\bm{e}_q});t \le \bm{e}_q|\cF_t}
    & = \ce^{-qt}\cbra*{\int_0^\infty f((x+X_t)\vee S_t)\bP(S_{\bm{e}_q}\in \cd x) } \\
    & = \ce^{-qt} \cbra*{f(S_t)
     \widetilde{\bP}(\widetilde{S}_{\widetilde{\bm{e}}_q}\le S_t-X_t)
     +\int_{S_t}^\infty f(x)
     \widetilde{\bP}(\widetilde{S}_{\widetilde{\bm{e}}_q}\in \cd x-X_t) }.
  \end{align}
  It follows from~\eqref{eq:PSeq} and~\eqref{eq:PSeq-dx} that
  \begin{align}\label{eq:exp_first_raw_asconv}
    N_t^{(q, f)}
    &=\ce^{-qt}\cbra*{h_q(S_t-X_t)+\int_{S_t}^\infty f(x)
    \int_0^\infty \ce^{-qs}V(\cd s, \cd x-X_t)}.
  \end{align}
  Hence, by the monotone convergence theorem, it holds that
  \begin{align}
    N_t^{(q, f)}
    \xrightarrow[q\down 0]{}
    f(S_t)h(S_t-X_t) + \int_{S_t}^\infty f(x)V([0,\infty), \cd x-X_t)
    = M^{(f)}_t,
    \quad \text{\bP-a.s.}
  \end{align}
  Since \(N_0^{(q, f)}=M_0^{(q, f)}\), we also have \(\lim_{q\down 0}
  M_0^{(q, f)}=M_0^{(f)}\), \(\bP\)-a.s.

  Plus, we have
  \begin{align}
    M_t^{(q, f)}-N_t^{(q,f)}=\bP\sbra{f(S_{\bm{e}_q}); \bm{e}_q \le t|\cF_t}
    = q\int_0^t \ce^{-qs} f(S_s) \, \cd s.\label{eq:At}
  \end{align}
  For \(x>0\), it holds that
  \begin{align}
    \frac{q}{\kappa(q,0)}
    = \frac{\bP(S_{\bm{e}_q} \le x)}{\kappa(q,0)} \frac{q}{\bP(S_{\bm{e}_q} \le x)}.
  \end{align}
  By the monotone convergence theorem, we have
  \begin{align}
    \frac{\bP(S_{\bm{e}_q}\le x)}{q}
     & = \bP\sbra*{\int_0^\infty \ce^{-qt} 1_{\cbra{S_t \le x}}\,\cd t}
     \xrightarrow[q\down 0]{} \bP\sbra*{\int_0^\infty 1_{\cbra*{S_t\le x}}\, \cd t}
    = \bP\sbra{T_{(x, \infty)}}.
  \end{align}
  Since \(X\) oscillates, it is known that
  \(\bP\sbra{T_{(x, \infty)}} = \infty\),
  see, e.g.,~\cite[Proposition VI.17]{MR1406564}.
Combining this with~\eqref{eq:PSeq}, we obtain
  \begin{align}
    \frac{q}{\kappa(q,0)} \xrightarrow[q\down 0]{}\frac{h(x)}{\bP\sbra{T_{(x, \infty)}}}
    = 0.\label{eq:q/kq}
  \end{align}
  By~\eqref{eq:At} and~\eqref{eq:q/kq}, it holds that
  \begin{align}
   M_t^{(q, f)}-N_t^{(q,f)}
   =\frac{q}{\kappa(q,0)}\int_0^t \ce^{-qs} f(S_s) \, \cd s
    \xrightarrow[q\down 0]{}
    0,
    \quad \text{\bP-a.s.}
  \end{align}
  We proceed the proof of the \(L^1(\bP)\) convergence.
  By Fatou's lemma and since \((M_t^{(q, f)}, t\ge 0)\) is a martingale we have
  \begin{align}
    M_0^{(f)} =\bP\sbra{M_t^{(f)}} \le
    \liminf_{q\down 0}\bP\sbra{N_t^{(q,f)}}\le  \liminf_{q\down 0}\bP\sbra{M_t^{(q,f)}}
    =  \liminf_{q\down 0}M_0^{(q, f)}= M_0^{(f)}.
  \end{align}
  Thus, we may apply Scheff\'{e}'s lemma to deduce the \(L^1(\bP)\) convergence.
\end{proof}

\section{The result with constant clock}\label{Sec:const-clock}
To prove Theorem~\ref{Thm:const-result},
we utilize the following two lemmas.
\begin{Lem}[{\cite[Theorem 2]{MR3531705}}]\label{Lem:SM-conv}
  Suppose the conditions~\ref{cond:A}--\ref{cond:C} and~\ref{cond:D1}--\ref{cond:D3}
  hold.
  Then, it holds that
  \begin{align}
    \lim_{t\to\infty} \frac{\varphi_t(x)}{n(t < \zeta)} =
    h^{\prime}(x)\label{eq:lem-SM-conv}
  \end{align}
  uniformly in \(x\) on every compact subset of \((0, \infty)\).
\end{Lem}

\begin{Lem}[{\cite[Theorem 3 and its proof]{MR3531705}}]\label{Lem:SM-ineq}
  Suppose the conditions~\ref{cond:A}--\ref{cond:C} and~\ref{cond:D1}--\ref{cond:D3}
  hold.
  Then, for fixed \(x_0, t_0 > 0\), there exists a constant \(C_1 > 0\) such that
  \begin{align}
    \varphi_t(x) \le C_1 h^{\prime}(x) n(t < \zeta),
    \quad 0 < x \le x_0,\, t \ge t_0.\label{eq:lem-SM-ineq}
  \end{align}
  Additionally, there exists a constant \(C_2 > 0\) such that
  \begin{align}
    \varphi_t(x) \le C_2 (h^{\prime}(x) \vee h(x)) n(t < \zeta), \quad x > 0,\, t \ge
    t_0.
  \end{align}
\end{Lem}
From these, we deduce the following lemma.

\begin{Lem}\label{Lem:conv-h}
  Suppose the conditions~\ref{cond:A}--\ref{cond:C} and~\ref{cond:D1}--\ref{cond:D3}
  hold.
  Then, for all \(x \ge 0\), we have
  \begin{align}
    \lim_{t\to\infty} \frac{\bP(S_t \le x)}{n(t < \zeta)} = h(x).
  \end{align}
\end{Lem}
\begin{proof}
  Integrating~\eqref{eq:lem-SM-conv} on \([\delta, x]\subset (0,\infty)\),
  we obtain
  \begin{align}\label{eq:lem-conv-with-delta}
    \lim_{t\to\infty} \frac{\bP(\delta \le S_t \le x)}{n(t < \zeta)} = h(x) - h(\delta)
  \end{align}
  Thus we have
  \begin{align}
    \liminf_{t\to\infty} \frac{\bP(S_t \le x)}{n(t < \zeta)} \ge h(x) - h(\delta).
  \end{align}
  Letting \(\delta \down 0\), we obtain
  \begin{align}
    \liminf_{t\to\infty}  \frac{\bP(S_t \le x)}{n(t < \zeta)} \ge h(x).
  \end{align}
  On the other hand, by integrating~\eqref{eq:lem-SM-ineq}, we obtain
  \begin{align}\label{eq;lem-bdd}
    \bP(S_t \le x) \le C_1 h(x) n (t < \zeta), \quad 0 \le x \le x_0, \, t_0 \le t,
  \end{align}
  here we use the fact that
  \(h(0) = 0\) and \(\bP(S_t = 0 ) = 0 \) for all \(t>0\).
  Suppose now that
  \begin{align}
    \limsup_{t\to\infty} \frac{\bP(S_t \le x)}{n(t < \zeta)} = h(x) + \alpha.
  \end{align}
  Then by~\eqref{eq:lem-conv-with-delta}, for \(0 < y \le x\),
  we have
  \begin{align}
    \limsup_{t\to\infty} \frac{\bP(S_t \le y)}{n(t < \zeta)} = h(y) + \alpha.
  \end{align}
  Since~\eqref{eq;lem-bdd} and \(h(0)=0\), we deduce that \( \alpha = 0\).
  Hence we obtain the desired result.
\end{proof}

\begin{proof}[Proof of Theorem~\ref{Thm:const-result}]
  Using the Markov property, we obtain, for \(s>t\),
  \begin{align}
    \bP\sbra{f(S_s)|\cF_t}
     & = \widetilde{\bP}\sbra{f(S_t \vee (\widetilde{S}_{s-t}+X_t))} \\
     & = \int_0^\infty f(S_t \vee (x+X_t)) \varphi_{s-t}(x)\, \cd
    x                                                                \\
     & = f(S_t) \widetilde{\bP}\rbra{\widetilde{S}_{s-t}\le S_t -
      X_t} +
    \int_{S_t}^\infty f(x)\varphi_{s-t}(x-X_t) \, \cd x.
  \end{align}
  Hence, by Lemmas~\ref{Lem:SM-ineq} and~\ref{Lem:conv-h},
  and by the dominated convergence theorem,
  it holds that
  \begin{align}
    \frac{\bP\sbra{f(S_s)|\cF_t}}{n(s<\zeta)}
     & = \frac{n(s-t<\zeta)}{n(s<\zeta)}
    \cbra*{f(S_s)\frac{\widetilde{\bP}\rbra{\widetilde{S}_{s-t} \le
          S_t-X_t}}{n(s-t<\zeta)}
    + \int_{S_t}^\infty f(x)\frac{\varphi_{s-t}(x-X_t)}{n(s-t<\zeta)}\, \cd x} \\
     & \xrightarrow[s\to\infty]{} f(S_t) h(S_t-X_t) + \int_{S_t}^\infty f(x) h^\prime (x) \,\cd x
     = M_t^{(f)},
  \end{align}
  \bP-a.s. From Scheff\'{e}'s lemma,
  the convergence is also holds in \(L^1(\bP)\), and
  the proof is now complete.
\end{proof}

\noindent\textbf{Acknowledgments} The author would like to express
his deep gratitude to Professor
Kouji Yano for his helpful advice and encouragement. This work was supported by JSPS
Open Partnership Joint Research Projects grant no. JPJSBP120249936.

\end{document}